\documentclass[11pt,leqno]{article}
\usepackage{verbatim}
\usepackage[english]{babel}
\usepackage{enumitem}
\usepackage{tikz}
\usepackage{graphicx}
\usepackage{booktabs,tabularx}
\usepackage{verbatim}
\usepackage{etoolbox}
\usepackage{tkz-euclide}
\usepackage{geometry}
\usepackage{mathtools}
\usepackage{color}
\usepackage[title]{appendix}
\usepackage[all]{xy}
\usepackage{tikz-cd}
\usepackage{blindtext}
\usepackage[T1]{fontenc}
\usepackage{amsthm}
\usepackage{amsfonts}
\usepackage{txfonts}
\usepackage{palatino,amssymb,epsfig}
\usepackage{latexsym,epsf,epic,amscd}
\usepackage{mathrsfs}
\usepackage{graphicx}
\usepackage{caption}
\usepackage{subcaption}

\usepackage{appendix}
\usepackage{amsmath}
\usepackage{bookmark} 
\usepackage[nottoc]{tocbibind}
\hypersetup{
colorlinks=true,
linkcolor=black,
citecolor=black,
anchorcolor=black,
urlcolor=black}
\allowdisplaybreaks[4]
\numberwithin{equation}{section}
\DeclareFontFamily{OMX}{yhex}{}
\DeclareFontShape{OMX}{yhex}{m}{n}{<->yhcmex10}{}
\DeclareSymbolFont{yhlargesymbols}{OMX}{yhex}{m}{n}
\DeclareMathAccent{\wideparen}{\mathord}{yhlargesymbols}{"F3}

\DeclareMathOperator{\area}{Area}
\DeclareMathOperator{\arccoth}{arccoth}
\DeclareMathOperator{\arctanh}{arctanh}

\newcommand{\Addresses}{{
  \bigskip
  \footnotesize
  \noindent Guangming Hu, \href{18810692738@163.com}{18810692738@163.com}
\newline\textit{ College of Science, Nanjing University of Posts and Telecommunications,
  Nanjing, 210003, P.R. China.}\par\nopagebreak
  \medskip
  \noindent Ziping Lei, \href{zplei@ruc.edu.cn}{zplei@ruc.edu.cn}
  \newline\textit{ School of Mathematics, Renmin University of China, Beijing, 100872, P.R. China.} \par\nopagebreak
  \medskip
  \noindent Yi Qi,
\href{yiqi@buaa.edu.cn}{yiqi@buaa.edu.cn}
\newline\textit{ School of Mathematical Sciences,
Beihang University, Beijing, 100191, P.R. China}\par\nopagebreak
    \medskip
  \noindent Puchun Zhou, \href{pczhou22@m.fudan.edu.cn}{pczhou22@m.fudan.edu.cn}
  \newline\textit{ School of Mathematical Sciences, Fudan University, Shanghai, 200433, P.R. China.} }}

\title{  Combinatorial  $p$-th Calabi Flows for Total Geodesic Curvatures in hyperbolic background geometry}
\author{ Guangming Hu, Ziping Lei, Yi Qi and Puchun Zhou}

\date{}
\newtheorem{theorem}{Theorem}[section]
\newtheorem{lemma}[theorem]{Lemma}
\newtheorem{proposition}[theorem]{Proposition}
\newtheorem{corollary}[theorem]{Corollary}
\theoremstyle{definition}
\newtheorem{definition}[theorem]{Definition}
\newtheorem{remark}[theorem]{Remark}
\newtheorem{example}[theorem]{Example}
\newcommand{\pp}[2]{\frac{\partial#1}{\partial#2}}
\newcommand{\ddt}[1]{\frac{\mathrm{d}#1}{\mathrm{d}t}}
\begin{document}
\maketitle

\begin{abstract}
In hyperbolic background geometry, we investigate  a generalized  circle packing (including circles, horocycles and hypercycles) with conical singularities on a surface with boundary, which has a  total geodesic curvature on each generalized circle of  this circle packing and a discrete Gaussian curvature  on the center of each dual circle. The purpose of this paper is to find this type of circle packings with prescribed total geodesic curvatures on generalized circles and discrete Gaussian curvatures on centers of dual circles. To achieve this goal, we firstly establish existence and rigidity on this type of circle packings by the variational principle. Secondly, for $p>1$, we introduce combinatorial $p$-th Calabi flows to find the circle packing with prescribed total geodesic curvatures on generalized circles and discrete Gaussian curvatures on centers of dual circles for the first time.

\end{abstract}

\section{Introduction}
\subsection{Background}
Given a smooth manifold $M$ and a function $K$ on $M$, it is a natural problem  whether there exists a metric $g$ on $M$ in a certain conformal class such that $K$ can be realized as the scalar curvature of $g$. When $M$ is the sphere $\mathbb{S}^2$, the problem is the classical Nirenberg problem, e.g. see \cite{chang1987prescribing}. When it comes to the polyhedral surface, the natural problem is to find the metric that admits prescribed discrete Gaussian curvature, which is $2\pi$ minus the sum of cone angles at the conical point. Besides, the circle packing metric of polyhedral surface is one of the discrete analogues of the conformal structure on smooth surfaces. For other discrete conformal structure, we recommend readers to read \cite{BPS,Luo,izmestiev}.
For finding the metric with prescribed Gaussian curvatures, flow approaches are developed in both smooth and discrete cases. 

The Ricci flow was introduced by Hamilton \cite{Hamilton} and Perelman improved  Hamilton’s program of Ricci flow to solve the Poincaré conjecture and Thurston’s geometrization conjecture \cite{Perelman1, Perelman2, Perelman3}. And it can be used to prove the uniformization theorem, see \cite{Hamilton,Tian}.
 To find metrics with constant curvatures in a certain K\"{a}hler class, Calabi studied the variational problem of Calabi energy and introduced a geometric flow, which is now well known as the Calabi flow, see \cite{Calabi,Chenxx_1}.  
 
For the discrete geometry, the combinatorial Ricci flow was introduced by Chow and Luo \cite{chow-Luo} for finding Euclidean (spherical or hyperbolic) polyhedral surfaces with zero discrete Gaussian curvatures, as a counterpart of  Ricci flows on smooth manifolds. They showed that the  solution of a combinatorial Ricci flow exists for the long time and converges under some combinatorial conditions in Euclidean and hyperbolic background geometry. There are many applications on combinatorial Ricci flows, for example, see \cite{Feng, Gehua, ge2, GeJiang, David, Gu1, Gu2, Gu3, Luo}.

For the Calabi flow in the discrete case, Ge in his Ph.D. thesis \cite{Ge1} introduced the combinatorial Calabi flow for circle packing metrics. Moreover, 
 Ge \cite{Gehua}, Ge and Hua \cite{Gehua2}, and Ge and Xu \cite{Gexu2} proved that the combinatorial Calabi flow exists for all the time and converges for the Euclidean (hyperbolic) circle packing under some combinatorial conditions. After that, Lin and Zhang \cite{Lin} introduced the
combinatorial $p$-th Calabi flows which precisely equal the
combinatorial Calabi flows  when $p = 2$. 
 Recently, the combinatorial   Calabi flow has been widely concerned on the polyhedral surfaces in Euclidean and hyperbolic background geometry, for more details see \cite{Bate2, Feng2, Luo2, wu, xu, xu2}.

In hyperbolic background geometry, this paper investigates  a generalized  circle packing (including circles, horocycles and hypercycles) with conicial singularities on a surface with boundary, which has a  total geodesic curvature on each generalized circle of  this circle packing and a discrete Gaussian curvature  on the center of each dual disk. To find the type of hyperbolic  circle packings with prescribed total geodesic curvatures on generalized circles and discrete Gaussian curvatures on centers of dual disks, we establish an existence and rigidity on this type of circle packings by the variational principle, and then introduce the $p$-th Calabi flow of total geodesic curvatures for finding the circle packing with prescribed total geodesic curvatures.
\subsection{Main results}

This paper continues the study of \cite{HHSZ} with respect to a generalized hyperbolic circle packing on a surface with boundary.  In \cite{HHSZ}, the partial authors of this paper introduced a generalized hyperbolic circle (including circles, horocycles and  hypercycles) packing on a surface with boundary  with respect to a total geodesic curvature on each interior vertex and each boundary vertex, whose contact graph is the $1$-skeleton of a finite polygonal cellular decomposition.
The total geodesic curvature was first introduced in the work of Nie \cite{nie} as a powerful tool for studying the rigidity of the circle patterns in the spherical background geometry.  In   \cite{BHS}, the first author of this paper established an existence and rigidity  of the generalized hyperbolic circle packing metric with conical singularities on a  triangulated surface with respect to a total geodesic curvature on each vertex.


%

\subsubsection{Generalized circle packings on polygons}\label{sec121}

A \textbf{generalized circle} on the Poincar\'{e} disk model $\mathbb{H}^2$ is a "circle" with radius $r$ and constant geodesic curvature $k$, which is either a circle, a horocycle or a hypercycle. As the generalization of  center of a circle,  the "center" of a horocycle is defined as the unique limit point on the boundary and the "center" of a hypercycle is defined as the geodesic connecting its two limit points on the boundary, as shown in Figure \ref{cycles}.

 \begin{figure}[htbp]
\centering
\includegraphics[scale=0.4]{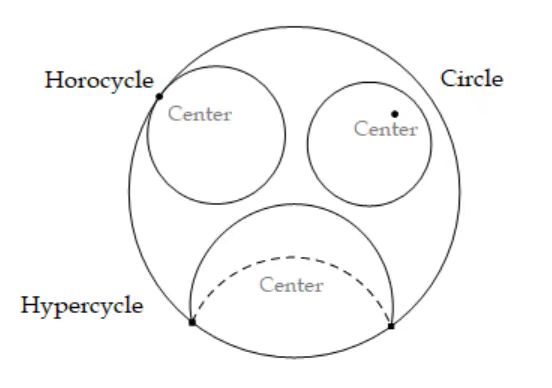}
\captionof{figure}{\small A circle, a horocycle and a hypercycle  on $\mathbb{H}^2$.}
  \label{cycles}
\end{figure} 

For an arc $C$ on a hypercycle $\Gamma$ with geodesic axis $\gamma$ and radius $r$, the inner angle $\theta$ of the arc $C$ is defined as the hyperbolic length of shortest distance projection of $C$ onto the geodesic axis $\gamma$, as shown in Figure \ref{hyperarc}.
 \begin{figure}[htbp]
\centering
\includegraphics[scale=0.4]{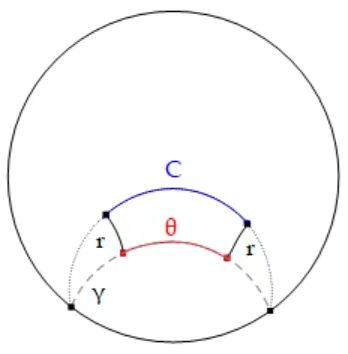}
\captionof{figure}{\small The inner angle $\theta$ of the arc $C$ in a hypercycle on $\mathbb{H}^2$.}
  \label{hyperarc}
\end{figure}

\begin{table}\centering
\centering
\footnotesize	
\begin{tabular}
{l|l|l|l}
 &  \small Radius &  \small Absolute value of geodesic curvature &  \small  Arc length \\ \hline
 \small Circle &  $0<r<\infty$&  $k=\coth r$&  $l=\theta\sinh r$  \\ 
 \small Horocycle &  $r=\infty$&  $k=1$&  -----  \\
 \small Hypercycle &  $0<r<\infty$& $k=\tanh r$ &  $l=\theta\cosh r$  \\ \hline
\end{tabular}

\caption{\small The relation of radii, absolute values of geodesic curvatures and arc lengths}


\label{relation}
\end{table}
The Table \ref{relation} shows the relations among the radius $r$ and absolute geodesic curvature $k$ of generalized circle, and the inner angle $\theta$ and the corresponding arc length $l$ of the generalized circle. For the case of horocycle, the arc length is shown in Lemma $2.5$ of \cite{BHS}. 
 
 \begin{figure}[htbp]
\centering
\includegraphics[scale=0.4]{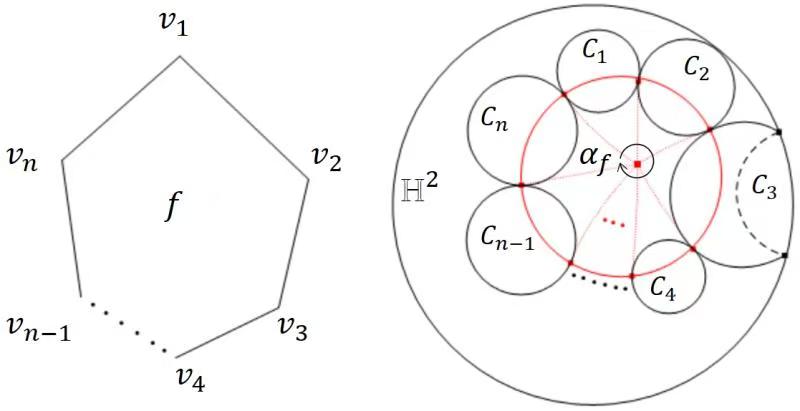}
\captionof{figure}{\small The polygon $f$ and generalized circle packing of $f$.}
  \label{cp}
\end{figure}

\begin{lemma}\label{packing}
        Given an abstract polygon $f$ with a vertex set $V_{f}$ consisting of $v_1,\cdots,v_n$. For $k=(k_1,\cdots,k_n)\in \mathbb{R}_{>0}^{n}$ on $V_{f}$ and $(2-n)\pi<Y_f<2\pi$, then there exists a geometric pattern $\mathcal{P}$ formed by $n$ generalized circles $C_1,\cdots,C_n$ on $\mathbb{H}^2$ which satisfies the following conditions:
\begin{enumerate}
    \item The generalized circles $C_{i}$ and $C_{j}$ is tangent to each other if they have a nonempty intersection.
    \item The generalized circle $C_i$ is tangent to another generalized circle $C_j$  if and only if $v_i$ and $v_j$ are joined by an edge of $f$.  
    \item There exists a hyperbolic circle $C_f$ with a conical angle $\alpha_{f}=2\pi-Y_f$ at its center which perpendiculars to each $C_i$ and contains all tangent points of  $C_1,\cdots,C_n.$ We call this circle the dual circle of $f$.
    \item The geodesic curvature of generalized circle $C_i$ is $k_i, i=1,\cdots, n$. 
\end{enumerate} 
The geometric pattern $\mathcal{P}$ is also called a \textbf{generalized circle packing} with geodesic curvatures $k=(k_1,\cdots,k_n)$ and conical angle $\alpha_f=2\pi-Y_f$ of the polygon $f$ as shown in Figure \ref{cp}. $Y_f$ is called the discrete Gaussian curvature at the center of dual circle $C_f$.

Moreover, let $k_f$ be the geodesic curvature of $C_f$, then $k_f$ is a $C^1$ continuous function with respect to $k=(k_1,\cdots,k_n).$
    \end{lemma}

\begin{proof}

For $k_0>1, k_i>0$, we can construct three types of perpendicular circles $C_0, C_i$ on $\mathbb{H}^2$ as shown in Figure \ref{quadrilaterals}, where $k_i$ is the geodesic curvature of $C_i, i=0,1,\cdots,n$. For each perpendicular circles $C_0, C_i$, connecting the centers of the two circles and the intersecting points by geodesic arcs, we can obtain a ''quadrilateral''. By $\theta_i$ we denote the inner angle of the quadrilateral at the center of $C_0$ in perpendicular circles $C_0, C_i$, as shown in Figure \ref{quadrilaterals}. $\theta_i$ is a smooth function of $(k_0,k_i)$ and by Lemma 2.1 in \cite{HHSZ}, we have 
\begin{equation}\label{ae1}\frac{\partial\theta_i}{\partial k_{0}}=\frac{2k_{0}k_{i}}{\sqrt{k_{0}^2-1}(k_{i}^2+k_{0}^2-1)}>0.
\end{equation}
Moreover, fixing $k_{i}$, we have
\begin{equation}\label{ae2}
        \lim_{k_{0}\rightarrow 1^+}\theta_i(k_{0},k_{i})=0,~~
        \lim_{k_{0}\rightarrow +\infty}\theta_i(k_{0},k_{i})=\pi.
\end{equation}
Gluing all the quadrilaterals together at the center of $C_0$, we can get a hyperbolic cone metric with a cone angle $\alpha(k_0,k)=\sum_{i=1}^{n}\theta_i(k_0,k_i).$ By (\ref{ae1}) and (\ref{ae2}), the cone angle $\alpha$ varies from $0$ to $n\pi$ continuously as $k_0$ varies in $(1,+\infty)$ and is strictly increasing with respect to $k_0$. 

Hence for $k=(k_1,\cdots,k_n)\in \mathbb{R}_{>0}^{n}$ on $V_{f}$ and $(2-n)\pi<Y_f<2\pi$, there exists a unique $k_f\in (1,+\infty)$ such that $\alpha(k_f,k)=\alpha_f=2\pi-Y_f$. By $C_f$ we denote the hyperbolic circle with geodesic curvature $k_f$. By the above argument, we obtain a geometric pattern $\mathcal{P}$ formed by $n$ generalized circles $C_1,\cdots,C_n$ on $\mathbb{H}^2$ which satisfies the conditions in the above lemma.

By the implicit function theorem, $k_f=k_f(k_1,...,k_n)$ is a smooth function with respect to $k=(k_1,...,k_n)$.
\end{proof}

\begin{remark}
    In Figure \ref{quadrilaterals}, from left to right, the geodesic curvature $k_i$ is smaller, equal or larger than $1$ respectively. Moreover, when $0<k_i<1$, the ``quadrilateral'' is actually a pentagon. However, we regard the ``center'' of a hypercycle as a point. Hence the pentagon is also called a ``quadrilateral''. 
\end{remark}

\begin{figure}[htbp]
\centering
\includegraphics[scale=0.40]{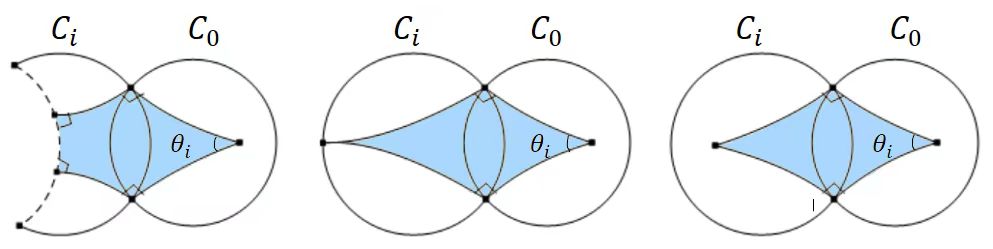}
\captionof{figure}{\small Three types of perpendicular circles $C_0, C_i$ and quadrilaterals.}
  \label{quadrilaterals}
\end{figure}

    By Lemma \ref{packing}, the generalized circle packing $\mathcal{P}$ on a polygon $f$ is determined by the geodesic curvature $k_{i}$ of the generalized circle corresponding to each vertex $v_{i}$ and the discrete Gaussian curvature $Y_f$ at the center of dual circle $C_f$.
    
Each vertex $v_{i}$ corresponds to one of following three objects in $\mathbb{H}^2$ according to its geodesic curvature $k_{i}$, i.e.

\begin{enumerate}
    \item The vertex $v_{i}$ corresponds to an interior point in $\mathbb{H}^2$ if the geodesic curvature $k_{i}>1$. 
    \item The vertex $v_{i}$ corresponds to an ideal point on the boundary of $\mathbb{H}^2$ if the geodesic curvature $k_{i}=1$. 
    \item The vertex $v_{i}$ corresponds to a geodesic arc in $\mathbb{H}^2$ if the geodesic curvature $0<k_{i}<1$. 
\end{enumerate}

\begin{figure}[htbp]
\centering
\includegraphics[scale=0.4]{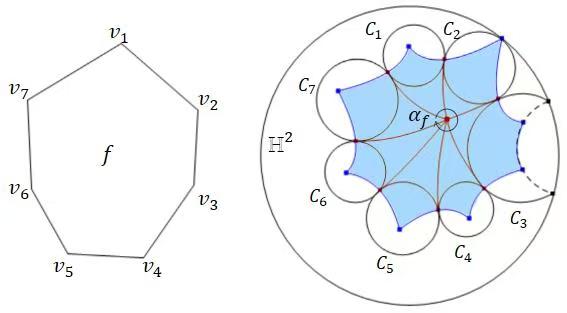}
\captionof{figure}{\small The polygon $f$ and hyperbolic geometric polygon $\hat{f}$ with a conical angle $\alpha_f$.}
  \label{hypolygon}
\end{figure} 

By $v_i$ we also denote the corresponding object in $\mathbb{H}^2$ of the vertex $v_i$. Given an abstract polygon $f$ with a vertex set $V_{f}$ consisting of $v_1,\cdots,v_n$, $k=(k_1,\cdots,k_n)\in \mathbb{R}_{>0}^{n}$ on $V_{f}$ and $(2-n)\pi<Y_f<2\pi$, by the proof of Lemma \ref{packing}, we obtain a hyperbolic geometric polygon $\hat{f}$ with a conical angle $\alpha_f=2\pi-Y_f$ corresponding to the generalized circle packing $\mathcal{P}$ of $f$ as shown in Figure \ref{hypolygon}. Let $v_{i}v_{j}$ be an edge of $f$ joining $v_i$, $ v_j\in V_f$. The hyperbolic length $l_{ij}$ of geodesic segment $\gamma_{ij}$  of $v_{i}v_{j}$ belongs to one of the following four cases.
\begin{enumerate}
    \item If  $v_i $ and $v_j$ are two interior points in $\mathbb{H}^2$, then the hyperbolic length is  $$l_{ij}=\arccoth k_i+\arccoth k_j.$$ 
    \item If $v_i $ is an interior point in $\mathbb{H}^2$ and $v_j$ is a geodesic arc in $\mathbb{H}^2$, then the hyperbolic length is
    $$
    l_{ij}=\arccoth k_i+\arctanh k_j.
    $$ 
 \item If at least one of $v_i$ and $v_j$ is an ideal point  on the boundary of $\mathbb{H}^2$, then the hyperbolic length is  $$l_{ij}=+\infty.$$
    \item If $v_i$ and $v_j$ are both geodesic arcs in $\mathbb{H}^2$, then the hyperbolic length is  $$l_{ij}=\arctanh{k_i}+\arctanh{k_j}.$$
    \end{enumerate}
In this paper, the above hyperbolic polygon $\hat{f}$ with a conical angle $2\pi-Y_f$ is called \textbf{generalized conical polygon} of  \textbf{generalized circle packing} with geodesic curvatures $k=(k_1,\cdots,k_n)$ on the polygon $f$.

\subsubsection{Generalized circle packing on surfaces}

Let $S$ be a surface with or without boundary. Let $\Sigma=(V,E,F)$ be a finite polygonal cellular decomposition of $S$, where $V,E$ and $F$ are sets of vertices, edges and faces, respectively. Denote $N(f)$ as the number of the vertices of $f\in F$.

\begin{definition}[Generalized  circle packing]
For a surface $S$ with a polygonal decomposition $\Sigma=(V,E,F)$, given a discrete Gaussian curvature $Y_f\in ((2-N(f))\pi,2\pi)$ for each $f\in F$ and geodesic curvatures $k\in \mathbb{R}^{|V|}_{>0}$ on $V$, we can construct a new surface $\hat{S}$ as follows.
\begin{enumerate}
    \item For each face $f\in F$ with $n$ vertices $v_1,\cdots,v_n$, by section \ref{sec121}, we can construct a 
generalized conical polygon $\hat{f}$ with a conical angle $\alpha_f=2\pi-Y_f$ of generalized circle packing with geodesic curvatures $(k_{v_1}, \cdots, k_{v_n})$ on the face $f$.
\item Gluing all the generalized conical polygons together along their edges, then we obtain a surface $\hat{S}$ with a hyperbolic conical metric, the metric is called a \textbf{generalized circle packing metric}.
\end{enumerate}
\end{definition}
Let $V_1,V_2$ and $V_3$ be the sets of vertices defined as follows.
\begin{enumerate}
    \item $v\in V_1$ iff $0<k_{v}<1$.
    \item $v\in V_2$ iff $k_{v}=1$.
    \item $v\in V_3$ iff $k_{v}>1$.
\end{enumerate}
Topologically, $\hat{S}$ can be obtained by removing a disk or a half disk (a point) for each $v\in V_1$($v\in V_2$) from $S$ respectively. 

\begin{example}
\begin{figure}[htbp]
\centering
\includegraphics[scale=0.4]{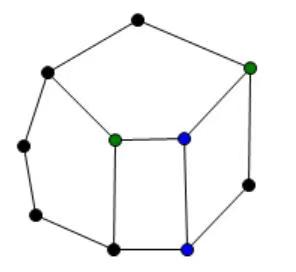}
\captionof{figure}{\small A surface $S$ with a polygonal decomposition.}
  \label{surface}
\end{figure}

    Let $S$ be a surface with polygonal cellular decomposition as shown in Figure \ref{surface}. We have $Y_f\in ((2-N(f))\pi,2\pi)$ for each $f\in F$ and $k\in \mathbb{R}_{>0}^{|V|}$ are discrete Gaussian curvatures on faces and geodesic curvatures on vertices respectively. Assuming the geodesic curvatures on green points are less than 1, the geodesic curvatures on blue points are equal to 1 and the geodesic curvatures on black points are greater than 1, the corresponding hyperbolic surface $\hat{S}$ with generalized circle packing metric  is shown in Figure \ref{surface_metric}.
\end{example}
 \begin{figure}[htbp]
\centering
\includegraphics[scale=0.3]{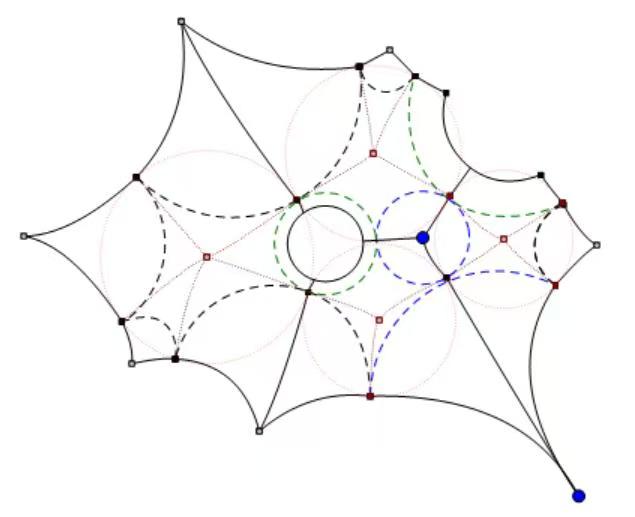}
\captionof{figure}{\small A surface $\hat{S}$ with generalized circle packing metric.}
  \label{surface_metric}
\end{figure}
 For a subset $W\subset V$, let $F_W$  denote the set
\[
F_W = \{f \in F~|~ \exists v\in W \text{ such that }v \text{ is a 
 vertex of } f \}.
\]
For $f\in F_W$, we define $$N(f,W)=\sharp\{v\in W~|~ v \text{ is a vertex of } f\}.$$

\begin{theorem}\label{main1}Set  $(S,\Sigma)$ a surface with or without boundary, which has a finite polygonal cellular decomposition. Denote by $V$, $E$ and $F$ the sets of vertices, edges and faces, respectively. Let $F_W$ be the set of faces having at least one vertex in $W$ for a subset $W\subset V$. Then there exists a surface $\hat{S}$ with generalized circle packing metric defined as above having total geodesic curvatures $L_1,\cdots,L_{\vert V\vert}$ on generalized circles and discrete Gaussian curvatures $Y_{1}, \cdots, Y_{|F|}$ on centers of dual circles if and only if $Y_{f}\in((2-N(f))\pi,2\pi)$, for each $f\in F$
and $(L_1,\cdots,L_{\vert V\vert})^T\in\mathcal{L}$, where \[\mathcal{L}=\left\{(L_1,\cdots,L_{\vert V\vert})^T\in\mathbb{R}^{\vert V\vert}_{>0}\left\vert\right.  \sum_{w\in W}L_v<\sum_{f\in F_W}\pi \min\bigg\{N(f,W),N(f)-2+\frac{Y_{f}}{\pi}\bigg\},~~\forall  W\subset
 V\right\}.\]
Moreover, the generalized circle packing metric is unique if it exists. 
\end{theorem}
 
\subsubsection{Combinatorial p-th Calabi  flow}


Motivated by the method of Lin and Zhang \cite{Lin}, we construct a combinatorial $p$-th Calabi flow with respect  to the total geodesic curvatures in hyperbolic background geometry as follows.

Let $g: V \rightarrow \mathbb{R}$ be a function on $ V$. For $p>1$, we construct the discrete $p$-th Laplace operator $\Delta_p$ on $g$, i.e.
\begin{align}\label{lap} \Delta_p g_i=\sum_{j\sim i}A_{ij}|g_j-g_i|^{p-2}(g_j-g_i),~~\forall i\in V,
\end{align}
where $i\sim j$ denotes the vertices $i,j$ adjacent and
$$A_{ij}=-\frac{\partial(\sum_{f\in F_{\{i\}}}L_{i,f})}{\partial k_j}k_j, ~~i\sim j, i\ne j,$$ $F_{\{i\}}$ is the set of all faces in $F$ with the vertex $i$. Then we can define the following combinatorial $p$-th Calabi flow.

\begin{definition}[Combinatorial $p$-th Calabi flow]
\begin{align}
\frac{d k_i}{dt}=k_i(\Delta_p -K_i)(L_i-{\hat{L}}_i),~~\forall i\in V,\label{p-calabi_flow}
\end{align}
where
\begin{align}\label{Ki}
   K_i=-k_i\frac{\partial}{\partial k_i}(\sum_{f\in F_{\{i\}}}\text{Area}(\Omega_f)) 
\end{align}
and $\hat{L}=(\hat{L}_1,\cdots,\hat{L}_{|V|})^T$ is a given vector in $ \mathbb{R}_{>0}^{\vert V\vert}$ which represents the prescribed total geodesic curvatures.
\end{definition}

\textbf{Remark.} For $p=2$, the combinatorial $p$-th Calabi flow \ref{p-calabi_flow} is the usual combinatorial Calabi flow for the total geodesic curvatures in hyperbolic background geometry.  
   
\begin{theorem}[Main theorem]\label{flowthm}
 Suppose $\hat{L}=(\hat{L}_1,\cdots,\hat{L}_{\vert V\vert})^T\in \mathbb{R}_{>0}^{\vert V\vert}$
is a vector defined on $V$. The following statements are
equivalent:
\begin{enumerate}
    \item $\hat{L}\in\mathcal{L}$, where $\mathcal{L}$ is defined in Theorem \ref{main1};

\item The solution of the combinatorial $p$-th Calabi flow (\ref{p-calabi_flow}) exists for all time $t\in [0,+\infty)$ and converges to a generalized circle packing metric with the total geodesic curvature $\hat{L}$.  
\end{enumerate}
 \end{theorem}
 \textbf{Organization.} We organize the paper as follows. In Section \ref{sec2}, we recall some facts for total geodesic curvatures on polygons and give similar conclusions on  generalized circle packing metric with a conical
angle. In section \ref{sec3}, we study the limit behavior of geodesic curvatures and prove Theorem \ref{main1}. In section \ref{sec4}, we prove the long time existence of combinatorial $p$-th Calabi flows \eqref{equi p-calabi_flow}. In section \ref{sec5}, we give some results and prove Theorem \ref{flowthm}.

\section{The potential functions for total geodesic curvatures on polygons}\label{sec2}
In this section, we consider the variational principle for total geodesic curvatures on polygons. Given an abstract polygon $f$ with a vertex set $V_{f}$ consisting of $v_1,\cdots,v_n$. For $k=(k_1,\cdots,k_n)\in \mathbb{R}_{>0}^{n}$ on $V_{f}$ and $(2-n)\pi<Y_f<2\pi$, by Lemma \ref{packing}, there exists a geometric pattern $\mathcal{P}$ formed by $n$ generalized circles $C_1,\cdots,C_n$ on $\mathbb{H}^2$ with geodesic curvatures $k=(k_1,\cdots,k_n)$ and conical angle $\alpha_f=2\pi-Y_f$. Then we have the following lemma.

\begin{lemma}\label{closed}
By $l_{i,f}$ we denote the length of the sub-arc $C_{i,f}$ of $C_i$ contained in dual circle $C_f$. Then the differential form
    \[
    \eta=\sum_{i=1}^{n}l_{i,f}\mathrm{d}k_i
    \]
     is closed.
\end{lemma}

\begin{proof}
By $l_f$ we denote the length of dual circle $C_f$ and by $l_{i,f}^*$ we denote the length of the sub-arc $C_{i,f}^*$ of $C_f$ contained in $C_i$ as shown in Figure \ref{2.1}. 

\begin{figure}[htbp]
\centering
\includegraphics[scale=0.40]{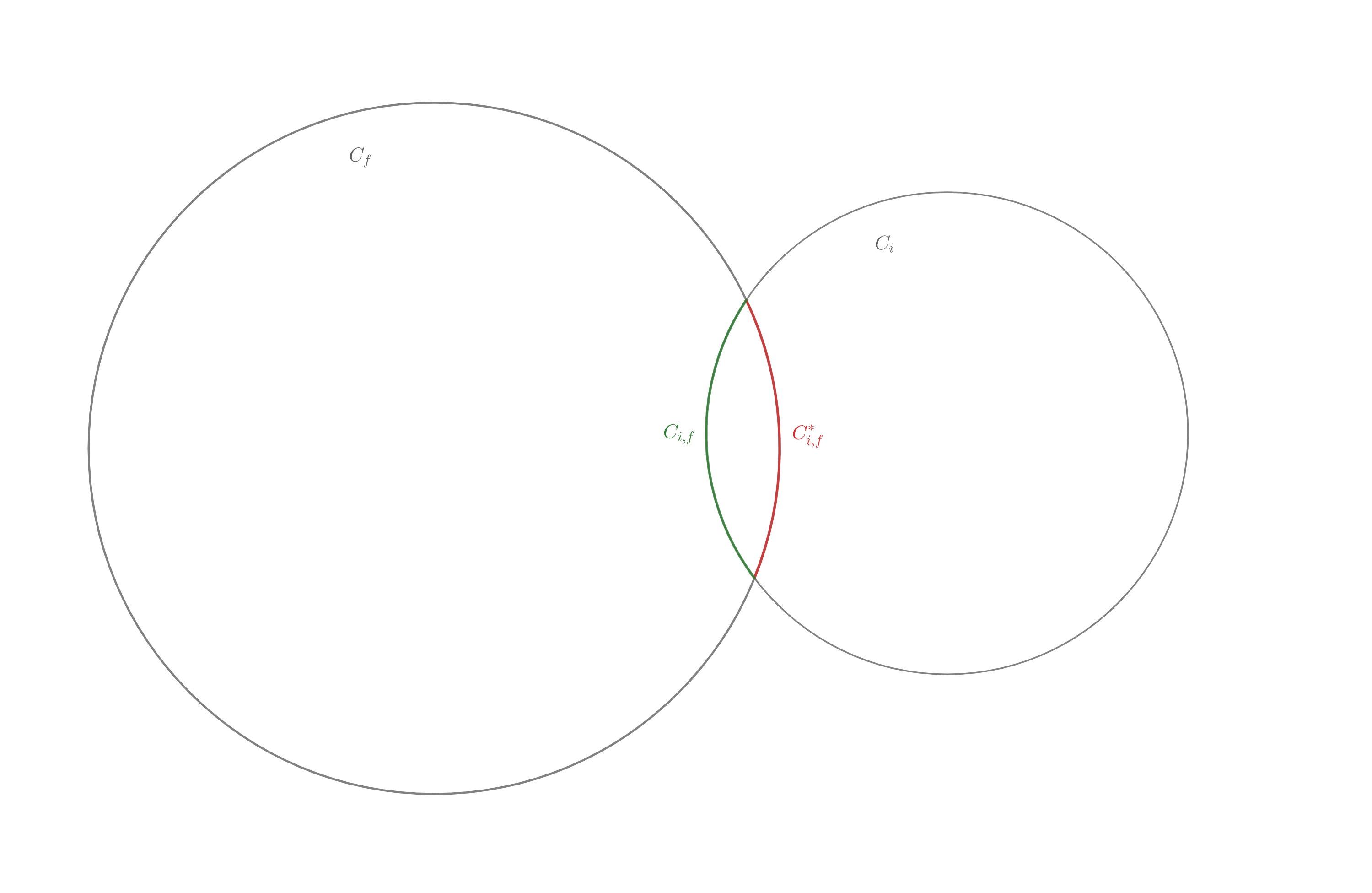}
\captionof{figure}{\small The sub-arc $C_{i,f}$ and $C_{i,f}^*$.}
  \label{2.1}
\end{figure} 

Then we have
$$\eta+l_fdk_f=l_fdk_f+\sum_{i=1}^{n}l_{i,f}\mathrm{d}k_i=\sum_{i=1}^{n}l_{i,f}\mathrm{d}k_i+\sum_{i=1}^{n}l_{i,f}^*\mathrm{d}k_f=\sum_{i=1}^{n}l_{i,f}\mathrm{d}k_i+l_{i,f}^*\mathrm{d}k_f.$$ By $k_f>1$ and Lemma 2.7 in \cite{BHS}, we know that $l_{i,f}\mathrm{d}k_i+l_{i,f}^*\mathrm{d}k_f$ is closed, $i=1,\cdots,n$. Since $l_f$ is the function with respect to $k_f$, $l_fdk_f$ is closed. Hence the 1-form $\eta$ is closed.   
\end{proof}
By $L_{i,f}=l_{i,f}\cdot k_i$ we denote the total geodesic curvature of sub-arc $C_{i,f}$ of $C_i$ contained in dual circle $C_f$. It is easy to know that $L_{i,f}$ is the smooth function with respect to $(k_1,\cdots ,k_n)$.
Set $s_i=\ln k_i$ and by Lemma \ref{closed}, we know that the form 
\[\omega_{f}=\sum_{i=1}^n L_{i,f} \mathrm{d}
s_{i},\]
is closed.
Hence we can define a potential function on $\mathbb{R}^n$ as follow:
\[
\mathcal{E}_f(s)=\int^{s}_{0}\omega_{f},\]
which is well-defined.

\begin{figure}[htbp]
\centering
\includegraphics[scale=0.40]{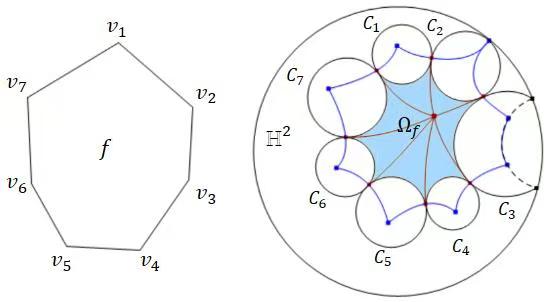}
\captionof{figure}{\small An example of polygon $f$ and region $\Omega_{f}$.}
  \label{area}
\end{figure} 

\begin{lemma}\label{gauss}
    Let $\Omega_f$ be the region with conical angle $0<\alpha_{f}<n\pi$ enclosed by sub-arcs $C_{i,f}, i=1,\cdots,n$ as shown in Figure \ref{area}. Then we have 
    \[
    \area(\Omega_f)=n\pi-\alpha_{f}-\sum_{i=1}^{n}L_{i,f}.
    \]
\end{lemma}
\begin{proof}
    This lemma can be directly deduced by the Gauss-Bonnet theorem.
\end{proof}

\begin{lemma}\label{area_change}
Given conical angle $0<\alpha_{f}<n\pi$, for geodesic curvatures $k=(k_1,\cdots,k_n)\in \mathbb{R}_{>0}^{n}$ on $V_{f}$, we have 
\begin{equation}\label{dd}
\pp{(\sum_{i=1}^{n}L_{i,f})}{k_j}>0, ~~ 1\leq j\leq n,
\end{equation}
and 
\[
\frac{\partial L_{i,f}}{\partial k_j}<0,~~\forall i\neq j.
\]
As a consequence, 
\[
\pp{\area{(\Omega_f})}{k_j}<0, ~~ 1\leq j\leq n.
\]
\end{lemma}
    \begin{proof}
By Lemma 2.5 in \cite{HHSZ}, we have 
\begin{align}\label{kiestimate}\frac{\partial(\sum_{i=1}^{n}L_{i,f})}{\partial k_j}=\left\{
\begin{array}{lc}
2(k_j^2-1)^{-\frac{3}{2}}(\frac{\sqrt{k_j^2-1}}{k_f}-\arctan{\frac{\sqrt{k_j^2-1}}{k_f}}),&k_j>1, \\\frac{2}{3k_f^3},&k_j=1,\\ 2(1-k_j^2)^{-\frac{3}{2}}(\arctanh{\frac{\sqrt{1-k_j^2}}{k_f}}-{\frac{\sqrt{1-k_j^2}}{k_f}}),&0<k_j<1.
\end{array}
\right.
\end{align}
Hence we know that $$\pp{(\sum_{i=1}^{n}L_{i,f})}{k_j}>0, ~~ 1\leq j\leq n.$$
By the proof of Lemma 2.8 in \cite{HHSZ}, we have that 
\[
\frac{\partial L_{i,f}}{\partial k_j}<0,~~\forall i\neq j.
\]
Moreover, we have
\begin{align}\label{bound}
0<-\pp{L_{i,f}}{k_j} <\sum_{m\neq j}-\pp{L_{m,f}}{k_j}<-\frac{2(k_f^2-1)}{k_f(1-k_f^2-k_j^2)},~~\forall i\neq j.
\end{align}
By Lemma \ref{gauss}, we obtain
$$
\pp{\area{(\Omega_f})}{k_j}=-\pp{(\sum_{i=1}^{n}L_{i,f})}{k_j}<0, ~~ 1\leq j\leq n.$$
\end{proof}

\begin{lemma}\label{convex}
The potential function $\mathcal{E}_f(s)$ is strictly convex.
\end{lemma}

\begin{proof}
By the definition of $\mathcal{E}_f(s)$, we have that
$$
\nabla\mathcal{E}_f=(L_{1,f},\cdots,L_{n,f})^T,~~\text{Hess}~\mathcal{E}_f=(\frac{\partial L_{i,f}}{\partial s_j})_{n\times n}. $$
By Lemma \ref{area_change}, we have that 
$$
\frac{\partial L_{i,f}}{\partial s_j}=k_j\frac{\partial L_{i,f}}{\partial k_j}<0,~~\forall i\neq j.
$$
Moreover, we have that
$$
\frac{\partial L_{j,f}}{\partial s_j}=\pp{(\sum_{i=1}^{n}L_{i,f})}{s_j}-\pp{(\sum_{i\ne j}L_{i,f})}{s_j}=k_j\pp{(\sum_{i=1}^{n}L_{i,f})}{k_j}-k_j\pp{(\sum_{i\ne j}L_{i,f})}{k_j}>0,~~1\leq j\leq n.$$
Then we obtain
$$\left\vert\frac{\partial L_{j,f}}{\partial s_j}\right\vert-\sum_{k\neq j}\left\vert\pp{L_{k,f}}{s_j}\right\vert=
\frac{\partial(\sum_{i=1}^nL_{i,f})}{\partial s_j}=k_j\frac{\partial(\sum_{i=1}^nL_{i,f})}{\partial k_j}>0,~~1\leq j\leq n.$$
By above argument, $\text{Hess}~\mathcal{E}_f$ is a symmetric strictly diagonally dominant matrix with positive diagonal entries, which yields that $\text{Hess}~\mathcal{E}_f$ is positive definite. Hence the potential function $\mathcal{E}_f(s)$ is strictly convex.
\end{proof}

\section{The existence and rigidity of generalized circle packing}\label{sec3}

\subsection{Limit behavior}

Given an abstract polygon $f$ with a vertex set $V_{f}=\{1,\cdots,n\}$ and discrete Gaussian curvature $(2-n)\pi<Y_f<2\pi$, we can construct a geodesic curvature sequence $\{k^m=(k_1^m,\cdots,k_n^m)\}_{m=1}^{+\infty}\subset \mathbb{R}_{>0}^{n}$ on $V_{f}$, by Lemma \ref{packing}, there exists a generalized circle packing sequence $\{\mathcal{P}^m\}_{m=1}^{+\infty}$ of the polygon $f$ with geodesic curvatures $\{k^m\}_{m=1}^{+\infty}$, dual circles $\{C_f^m\}_{m=1}^{+\infty}$ and conical angle $\alpha_f=2\pi-Y_f$, where the generalized circle packing $\mathcal{P}^m=\{C_i^m\}_{i=1}^n$. 

By $L_{i,f}^m$ we denote the total geodesic curvature of sub-arc $C_{i,f}^m$ of $C_i^m$ contained in dual circle $C_f^m$. Let $k^*=(k^*_1,\cdots,k^*_n)\in [0,+\infty]^n$ be a non-negative (possible infinity) vector. Then we study the limit behavior of $L_{i,f}^m$ as $k^m\rightarrow k^* (m\to +\infty)$.

\begin{lemma}\label{limit01}
If $k^m\rightarrow 0 (m\to +\infty)$, then $\lim_{m \to +\infty}L_{i,f}^m=0$, $i=1,\cdots,n$.
\end{lemma}

\begin{proof}
By $\theta_i^m$ we denote the angle of sub-arc $C_{i,f}^m$ at the center of $C_i^m$ and by $\theta_{i,f}^m$ we denote the angle of sub-arc $C_{i,f}^{m*}$ of $C_f^m$ contained in $C_i^m$ at the center of $C_f^m$ as shown in Figure \ref{3.1}.

\begin{figure}[htbp]
\centering
\includegraphics[scale=0.40]{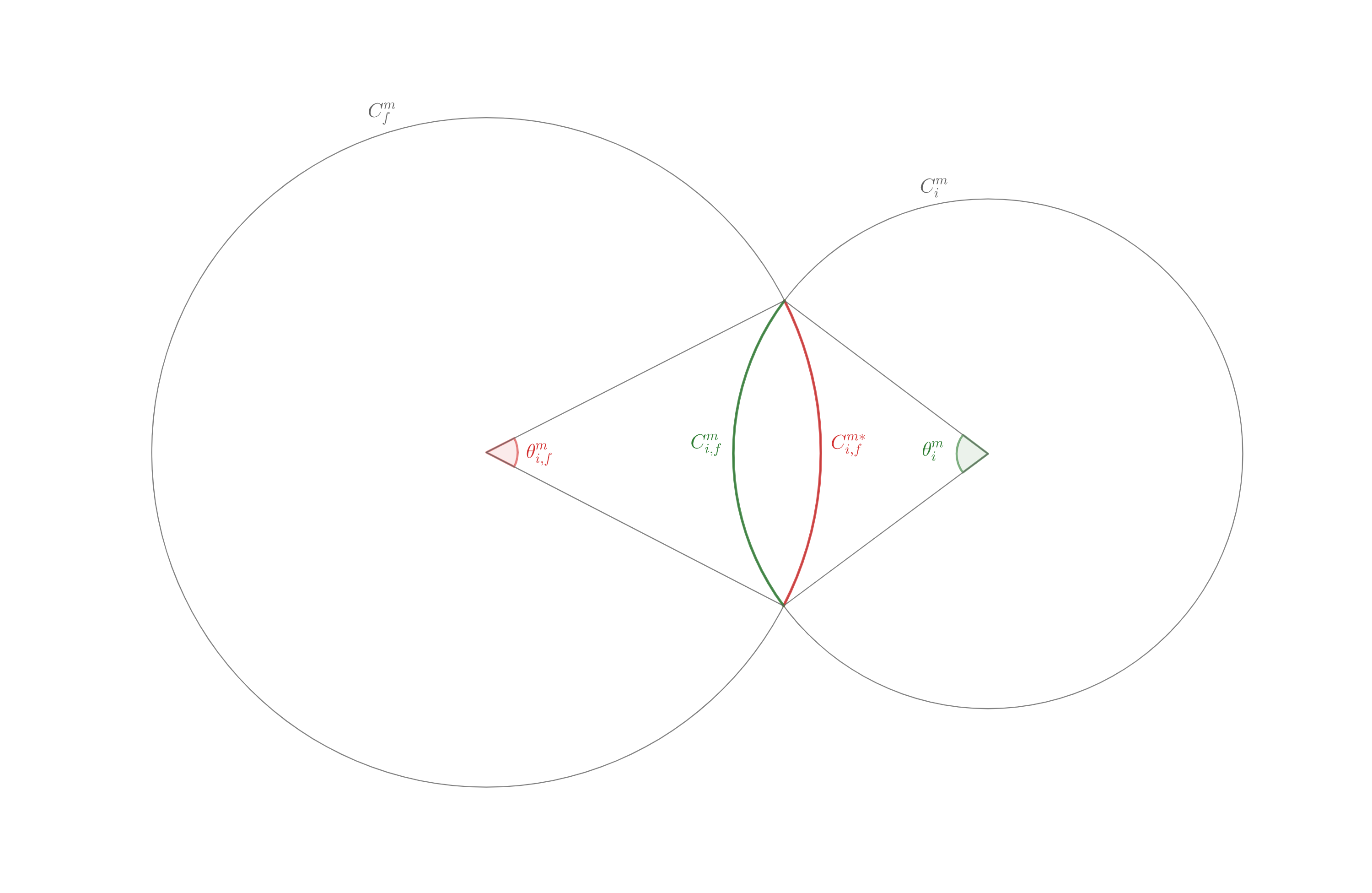}
\captionof{figure}{\small The angles $\theta_i^m$ and $\theta_{i,f}^m$ of sub-arcs $C_{i,f}^m$ and $C_{i,f}^{m*}$.}
  \label{3.1}
\end{figure} 

We first consider a special case of this lemma, we choose a geodesic curvature sequence $\{\tilde{k}^m=(\tilde{k}_1^m,\cdots,\tilde{k}_n^m)\}_{m=1}^{+\infty}$, where $\tilde{k}_i^m=\frac{1}{m+1}$, $i=1,\cdots,n$. By $\tilde{\theta}_i^m$ we denote the angle of sub-arc $C_{i,f}^m$, by $\tilde{\theta}_{i,f}^m$ we denote the angle of sub-arc $C_{i,f}^{m*}$ and by $\tilde{k}_f^m$ we denote the geodesic curvature of dual circle $C_f^m$ at this time. Then by symmetry we know that $\tilde{\theta}_{i,f}^m=\frac{\alpha_f}{n}$, $i=1,\cdots,n$. Since $0<\tilde{k}_i^m=\frac{1}{m+1}<1$, $m\ge 1$, by the law of the hyperbolic quadrilateral, we have 
$$
\cot\frac{\tilde{\theta}_{i,f}^m}{2}=\frac{\tilde{k}_i^m}{\sqrt{(\tilde{k}_f^m)^2-1}},
$$
then we obtain
\begin{equation}\label{e1}
\tilde{k}_f^m=\sqrt{\frac{\tan^2\frac{\tilde{\theta}_{i,f}^m}{2}}{(m+1)^2}+1}=\sqrt{\frac{\tan^2\frac{\alpha_f}{2n} }{(m+1)^2}+1}.
\end{equation}
Since $0<\tilde{k}_i^m=\frac{1}{m+1}<1$, $m\ge 1$, by hyperbolic trigonometric identities, we have
$$
\coth\frac{\tilde{\theta}_i^m}{2}=\frac{\tilde{k}_f^m}{\sqrt{1-(\tilde{k}_i^m)^2}},
$$
then we obtain
\begin{equation}\label{e2}
\tilde{\theta}_i^m=2\arccoth \frac{\tilde{k}_f^m}{\sqrt{1-(\tilde{k}_i^m)^2}}.
\end{equation}
By (\ref{e1}) and (\ref{e2}), we have
\begin{align*}
0<\tilde{\theta}_i^m&=2\arccoth \frac{\tilde{k}_f^m}{\sqrt{1-(\tilde{k}_i^m)^2}}\\
&=2\arccoth\sqrt{\frac{\frac{\tan^2\frac{\alpha_f}{2n} }{(m+1)^2}+1}{1-\frac{1}{(m+1)^2}}}\\
&=\ln \frac{\sqrt{1+\frac{\tan^2\frac{\alpha_f}{2n} }{(m+1)^2} }+\sqrt{1-\frac{1}{(m+1)^2} }}{\sqrt{1+\frac{\tan^2\frac{\alpha_f}{2n} }{(m+1)^2} }-\sqrt{1-\frac{1}{(m+1)^2} }}.
\end{align*}
Then we know that 
\begin{equation}\label{fg}
\begin{aligned}
0<\tilde{\theta}_i^m&=\ln \frac{\sqrt{1+\frac{\tan^2\frac{\alpha_f}{2n} }{(m+1)^2} }+\sqrt{1-\frac{1}{(m+1)^2} }}{\sqrt{1+\frac{\tan^2\frac{\alpha_f}{2n} }{(m+1)^2} }-\sqrt{1-\frac{1}{(m+1)^2} }}\\
&<\ln \frac{2+1}{\sqrt{1+\frac{\tan^2\frac{\alpha_f}{2n} }{(m+1)^2} }-\sqrt{1-\frac{1}{(m+1)^2} }}\\
&=\ln 3-\ln \bigg(\sqrt{1+\frac{\tan^2\frac{\alpha_f}{2n} }{(m+1)^2} }-\sqrt{1-\frac{1}{(m+1)^2} }\bigg)\\
&<\ln 3-\ln \bigg(\sqrt{1+\frac{\tan^2\frac{\alpha_f}{2n} }{(m+1)^2} }-1\bigg),
\end{aligned}
\end{equation}
as $m$ is sufficiently large.

By the Table \ref{relation} and (\ref{fg}), we have 
\begin{equation}\label{dg}
    0<\tilde{L}_{i,f}^m=\frac{\tilde{k}_i^m\tilde{\theta}_i^m}{\sqrt{1-(\tilde{k}_i^m)^2} } <\frac{1}{\sqrt{(m+1)^2-1} }\bigg(\ln 3-\ln \bigg(\sqrt{1+\frac{\tan^2\frac{\alpha_f}{2n} }{(m+1)^2} }-1\bigg)\bigg),~~i=1,\cdots,n,
\end{equation}
as $m$ is sufficiently large.

Since $k^m\rightarrow 0 (m\to +\infty)$, then $0<k_i^m<\frac{1}{m+1}=\tilde{k}_i^m$, $i=1,\cdots,n$ as $m$ is sufficiently large. By (\ref{dd}) and (\ref{dg}), for each $1\le i\le n$, we have 
$$
0<L_{i,f}^m\le\sum_{i=1}^nL_{i,f}^m\le\sum_{i=1}^n\tilde{L}_{i,f}^m\le\frac{n}{\sqrt{(m+1)^2-1} }\bigg(\ln 3-\ln \bigg(\sqrt{1+\frac{\tan^2\frac{\alpha_f}{2n} }{(m+1)^2} }-1\bigg)\bigg)\to 0(m\to +\infty).
$$ 
Hence we obtain $\lim_{m \to +\infty}L_{i,f}^m=0$, $i=1,\cdots,n$.
\end{proof}

\begin{lemma}\label{limit02}
    If $k^*_i=0$ for some $i$ $(1\leq i\leq n)$, then 
$\lim_{m\rightarrow +\infty}L_{i,f}^m=0$.
\end{lemma}
\begin{proof}
By $W$ we denote a subset of $V_f$, i.e. $W=\{i\in V_f~|~k^*_i=0\}$. For the geodesic curvature sequence $\{k^m\}_{m=1}^{+\infty}$, we construct another geodesic curvature sequence $\{\tilde{k}^m\}_{m=1}^{+\infty}$ such that $\tilde{k}_i^m=k_i^m, i\in W$, $\tilde{k}_j^m<k_j^m, j\in V_f-W$ and $\tilde{k}^m\to 0(m\to +\infty)$. 

For each $i\in W$ and $j\in V_f-W$, by Lemma \ref{area_change}, we know that $\frac{\partial L_{i,f}}{\partial k_j}<0$. Then $\forall i\in W$, we obtain $0<L_{i,f}^m<\tilde{L}_{i,f}^m$. By Lemma \ref{limit01}, we have
$$
0\le \lim_{m\to +\infty}L_{i,f}^m\le \lim_{m\to +\infty}\tilde{L}_{i,f}^m=0,~~\forall i\in W.
$$
Then we obtain $\lim_{m\rightarrow +\infty}L_{i,f}^m=0, \forall i\in W$. This completes the proof.
\end{proof}

\begin{lemma}\label{limita2}
 Let $I=\{i\in V_f~|~k^*_i=+\infty\}$, then 
    \begin{align}
     \lim_{m\to +\infty }\sum_{i\in I}L_{i,f}^m&=|I|\pi,~~\text{if}~~0\le|I|< n-2+\frac{Y_f}{\pi},\label{limit2} 
    \\\lim_{m\to +\infty}\sum_{i\in I}L_{i,f}^m&=(n-2)\pi+Y_f,~~\text{if}~~n-2+\frac{Y_f}{\pi}\le|I|\le n.\label{limit3}
     \end{align} 
Equivalently, we have
$$
 \lim_{m\to +\infty }\sum_{i\in I}L_{i,f}^m=\pi\min\bigg\{|I|, n-2+\frac{Y_{f}}{\pi}\bigg\}.
 $$
 \end{lemma}
\begin{proof}
If $0\le|I|< n-2+\frac{Y_f}{\pi}$, we claim that the geodesic curvature sequence $\{k_f^m\}_{m=1}^{+\infty}$ of dual circles $\{C_f^m\}_{m=1}^{+\infty}$ has an upper bound $\gamma<+\infty$. If not, by Lemma 2.1 in \cite{HHSZ}, we obtain 
$$
\lim_{m\to +\infty}\theta_{i,f}^m=\pi,~~\forall i\notin I.
$$
Hence we have 
$$
\lim_{m\to +\infty}\sum_{i\notin I}\theta_{i,f}^m=(n-|I|)\pi>\alpha_f=2\pi-Y_f,
$$
which causes a contradiction.

By above claim, we have
$$
\lim_{m\to +\infty}L_{i,f}^m=\lim_{m\to +\infty}\frac{k_i\theta_i^m}{\sqrt{k_i^2-1}}=\pi,~~\forall i\in I.
$$
Hence we obtain
$$
\lim_{m\to +\infty }\sum_{i\in I} L_{i,f}^m=|I|\pi.
$$

If $n-2+\frac{Y_f}{\pi}\le|I|\le n$, by Lemma \ref{gauss}, we have 
\begin{equation}\label{aq1}
\sum_{i\in I}L_{i,f}^m\le\sum_{i=1}^nL_{i,f}^m=(n-2)\pi+Y_f-\text{Area}(\Omega_f^m)<(n-2)\pi+Y_f,~~\forall m\ge 1,
\end{equation}
where $\Omega_f^m$ is the region with conical angle $\alpha_{f}=2\pi-Y_f$ enclosed by sub-arcs $C_{i,f}^m, i=1,\cdots,n$. We claim that
\begin{equation}\label{aq2}
\liminf_{m\to+\infty}\sum_{i\in I}L_{i,f}^m=(n-2)\pi+Y_f.
\end{equation}
If not, there exists a subsequence $\{k^{m_j}\}_{j=1}^{+\infty}$ of $\{k^m\}_{m=1}^{+\infty}$ such that 
\begin{equation}\label{as}
\lim_{j\to +\infty}\sum_{i\in I}L_{i,f}^{m_j}=\beta<(n-2)\pi+Y_f,
\end{equation}
where $\beta$ is a non-negative number. Then we claim that $k_f^{m_j}\to +\infty(j\to +\infty)$. If not, there exists a subsequence $\{k_f^{m_{j_s}}\}_{s=1}^{+\infty}$ of $\{k_f^{m_j}\}_{j=1}^{+\infty}$, which has an upper bound. By above argument, we have 
$$
\lim_{s\to +\infty}L_{i,f}^{m_{j_s}}=\pi,~~\forall i\in I.
$$
Then we obtain
$$
\lim_{s\to +\infty}\sum_{i\in I}L_{i,f}^{m_{j_s}}=|I|\pi\ge (n-2)\pi+Y_f,
$$
which causes a contradiction. 

By $r_f^{m_j}$ we denote the radius of circle $C_f^{m_j}$. Since $k_f^{m_j}\to +\infty(j\to +\infty)$, we know that $r_f^{m_j}\to 0$ and $\text{Area}(\Omega_f^{m_j})\to 0$ as $j\to +\infty$.

Moreover, we have 
$$
\lim_{j\to +\infty}L_{i,f}^{m_j}=\lim_{j\to +\infty}k_i^{m_j}\cdot l_{i,f}^{m_j}=0,~~\forall i\notin I,
$$
where $l_{i,f}^{m_j}$ is the length of sub-arc $C_{i,f}^{m_j}$.

Then by Lemma \ref{gauss}, we have
\begin{align*}
\lim_{j\to +\infty}\sum_{i\in I}L_{i,f}^{m_j}&=\lim_{j\to +\infty}\sum_{i\in I}L_{i,f}^{m_j}+\lim_{j\to +\infty}\sum_{i\notin I}L_{i,f}^{m_j}=\lim_{j\to +\infty}\sum_{i=1}^nL_{i,f}^{m_j}=\lim_{j\to +\infty}(n-2)\pi+Y_f-\text{Area}(\Omega_f^{m_j})\\
&=(n-2)\pi+Y_f,
\end{align*}
which contradicts (\ref{as}). Then by (\ref{aq1}) and (\ref{aq2}), we obtain
$$
\lim_{m\to +\infty}\sum_{i\in I}L_{i,f}^m=(n-2)\pi+Y_f.
$$
\end{proof}

\begin{corollary}\label{af}
For each $1\le i\le n$, we have $0<L_{i,f}<\pi$.
\end{corollary}

\begin{proof}
By Lemma \ref{area_change} and Lemma \ref{limita2}, we have 
$$
\frac{\partial L_{i, f}}{\partial k_i}>0, \lim _{k_i \rightarrow+\infty} L_{i, f}=\pi, \quad \forall 1 \leq i \leq n.
$$
This completes the proof.
\end{proof}

\subsection{The proof of Theorem \ref{main1}}The aim of this section is to obtain Theorem \ref{main1} by the  variational principle. We can define the potential function
\begin{equation}\label{W}
\mathcal{E}(s)=\sum_{f\in F} \mathcal{E}_{f}(s)\end{equation}
on $\mathbb{R}^{\vert V\vert}$, where $s_i$ is the value of the generalized circle packing metric at $i\in V$. We know that $\mathcal{E}_{f}(s)$ is strictly convex by Lemma \ref{convex}, hence $\mathcal{E}(s)$ is strictly convex. Note that
\[\frac{\partial \mathcal{E}(s)}{\partial s_i}=\sum_{f\in F_{\{i\}}}L_{i,f}\overset{\Delta}{=}L_i.\]
It is easy to see that $\nabla\mathcal{E} =(L_1,\cdots,L_{\vert V\vert})^T$, hence the Hessian of $\mathcal{E}$ equals to a Jacobi matrix, i.e.
\[\text{Hess}~\mathcal{E}=M=\begin{pmatrix}
 	 		\frac{\partial L_1}{\partial s_1}&\cdots&\frac{\partial L_{1}}{\partial s_{\vert V\vert}}\\
 	 		\vdots&\ddots&\vdots\\
 	 		\frac{\partial L_{\vert V\vert}}{\partial s_{1}}&\cdots&\frac{\partial L_{\vert V\vert}}{\partial s_{\vert V\vert}}\\
\end{pmatrix}.\]

\begin{proposition}\label{pro 3.1}
    The Jacobi matrix $M$ is positive definite.
\end{proposition}
\begin{proof}
 By Lemma \ref{gauss} and Lemma \ref{area_change}, we obtain
    $$
    \frac{\partial L_{i,f}}{\partial s_i}=-\frac{\partial\left(\operatorname{Area}\left(\Omega_f\right)+\sum_{j \neq i} L_{j,f}\right)}{\partial s_i}>0,~~\forall i\in V_f, \forall f\in F,
$$
and
$$
\frac{\partial L_{j,f}}{\partial s_i}< 0,~~\forall i\ne j\in V_f, \forall f\in F.
$$
Hence we have 
$$
\frac{\partial L_i}{\partial s_i}=\frac{\partial(\sum_{f\in F_{\{i\}}} L_{i,f})}{\partial s_i}=\sum_{f\in F_{\{i\}}}\frac{\partial L_{i,f}}{\partial s_i}>0,~~\forall i\in V.
$$
and
$$
\frac{\partial L_j}{\partial s_i}=\frac{\partial(\sum_{f\in F_{\{j\}}} L_{j,f})}{\partial s_i}=\sum_{f\in F_{\{j\}}}\frac{\partial L_{j,f}}{\partial s_i}\le 0,~~\forall i\ne j\in V.
$$
Hence we obtain

$$
\begin{aligned}
\left|\frac{\partial L_i}{\partial s_i}\right|-\sum_{j\ne i}\left|\frac{\partial L_j}{\partial s_i}\right|=\frac{\partial L_i}{\partial s_i}+\sum_{j\ne i}\frac{\partial L_j}{\partial s_i}=-\sum_{f\in F}\frac{\partial\operatorname{Area}(\Omega_f)}{\partial s_i}>0.
\end{aligned}
$$
Hence the Jacobi matrix $M$ is a strictly diagonally dominant matrix with positive diagonal entries, i.e. $M$ is positive definite.
\end{proof}

Now by Proposition \ref{pro 3.1}, the Hess $\mathcal{E}$ is positive definite, hence we have the following corollary.
\begin{corollary}\label{coro 3.2}
    The potential function $\mathcal{E}$ is strictly convex on $\mathbb{R}^{|V|}$.
\end{corollary}

The following property of convex functions plays a key role  in the proof of Theorem \ref{main1}.

\begin{lemma}\cite{Dai}\label{injective}
Let $\mathcal{E}:\mathbb{R}^n\to\mathbb{R}$ be a $C^2$-smooth strictly convex function with the positive definite Hessian matrix. Then its gradient $\nabla \mathcal{E}:\mathbb{R}^n\to\mathbb{R}^n$ is a smooth embedding.
\end{lemma}
Now we can proof the Theorem \ref{main1} as follows.
\begin{proof}[\textbf{Proof of Theorem \ref{main1}}] Let $\mathcal{E}$ be defined as (\ref{W}). We only need to proof the fuction $\nabla \mathcal{E}:\mathbb{R}^{\vert V\vert}\to \mathcal{L}$ is a homeomorphism.
For each $ s\in\mathbb{R}^{|V|}$, by the previous conclusion, we have 
$$
\nabla\mathcal{E}(s)=L(s)\overset{\Delta}{=}L.
$$
Then for each nonempty subset $W\subset V$, we have
$$
\sum_{w\in W}L_w=\sum_{w\in W}\sum_{f\in F_{\{w\}}}L_{w,f}=\sum_{f\in F_W}\sum_{w\in W\cap V_f}L_{w,f}.
$$
By Lemma \ref{gauss}, we have 
$$
\sum_{w\in W}L_w=\sum_{f\in F_W}\sum_{w\in W\cap V_f}L_{w,f}<\sum_{f\in F_W}(N(f)-2)\pi+Y_f.
$$
Besides, by Corollary \ref{af}, we have
$$
\sum_{w\in W}L_w=\sum_{f\in F_W}\sum_{w\in W\cap V_f}L_{w,f}<\sum_{f\in F_W}N(f,W)\pi.
$$
Then we obtain 
$$
\sum_{w\in W}L_w<\sum_{f\in F_W}\pi  \min\bigg\{N(f,W),N(f)-2+\frac{Y_{f}}{\pi}\bigg\},
$$
i.e. $L\in\mathcal{L}$. Hence $\nabla\mathcal{E}$ is a map from $\mathbb{R}^{|V|}$ to $\mathcal{L}$.

By Lemma \ref{injective}, we know that $\nabla\mathcal{E}$ is  a smooth embedding, hence we only need to prove that $\nabla \mathcal{E}(\mathbb{R^{|V|}})$ is equal to $\mathcal{L}$.
Since $\nabla \mathcal{E}$ is injective and its image is contained in $\mathcal{L},$
by Brouwer's Theorem on the Invariance of Domain, we need to analysis the boundary of its image. Choose a sequence $\{s^m\}_{m=1}^{+\infty}\subset \mathbb{R}^{|V|}$ such that
\[\lim_{m\to +\infty} s^{m}=a\in[-\infty,+\infty]^{\vert V\vert},\]
where $a_i=-\infty$ or $+\infty$ for some $i\in V$. Besides, we have 
$$
\nabla \mathcal{E}(s^m)=L(s^m)\overset{\Delta}{=}L^m,~~k^m_i\overset{\Delta}{=}\exp(s^m_i),~~\forall m\ge 1.
$$
Now we need to prove that $\{L^m\}^{+\infty}_{m=1}$ converges to the boundary of  $\mathcal{L}$. 

We define two subset $W, W'\subset V$, i.e.
$$
W=\{i\in V~|~a_i=+\infty\},~~W'=\{i\in V~|~a_i=-\infty\},
$$
where $W\ne \emptyset$ or $W'\ne \emptyset$. 

We first consider $W\ne \emptyset$. Then for each $i\in W$, $k^m_i=\exp(s^m_i)\to +\infty(m\to +\infty).$ Hence for each $f\in F_W$, by Lemma \ref{limita2}, we have
$$
 \lim_{m\to +\infty }\sum_{i\in W\cap V_f }L_{i,f}^m=\pi\min\bigg\{N(f,W), N(f)-2+\frac{Y_{f}}{\pi}\bigg\}.
 $$
Hence we have
$$
\lim_{m\to +\infty }\sum_{i\in W}L^m_{i}=\lim_{m\to +\infty }\sum_{f\in F_W}\sum_{i\in W\cap V_f}L_{i,f}^m=\sum_{f\in F_W}\pi\min\bigg\{N(f,W), N(f)-2+\frac{Y_{f}}{\pi}\bigg\},
$$
i.e. $\{L^m\}^{+\infty}_{m=1}$ converges to the boundary of  $\mathcal{L}$.

Now we consider $W'\ne \emptyset$. For each $ i\in W'$, $k^m_i=\exp(s^m_i)\to 0(m\to +\infty).$ Hence for each $f\in F_{W'}$, by Lemma \ref{limit01}, we have
$$\lim_{m\to +\infty}L_{i,f}^m=0,~~\forall i\in W'\cap V_f.$$
Now we obtain
$$
0\le\lim_{m\to +\infty }\sum_{i\in W'}L^m_{i}=\lim_{m\to +\infty }\sum_{f\in F_{W'}}\sum_{i\in W'\cap V_f}L_{i,f}^m=0,
$$
i.e. $\lim_{m\to +\infty }L^m_{i}=0, \forall i\in W'.$ Hence $\{L^m\}^{+\infty}_{m=1}$ converges to the boundary of  $\mathcal{L}$. 
We completed the proof.
\end{proof}

\section{Long time existence of combinatorial p-th Calabi flows}\label{sec4}
We use the change of variables $s_i=\ln k_i, \forall i\in V$, then the combinatorial p-th Calabi flow (\ref{p-calabi_flow}) has the following equivalent form, i.e.
\begin{definition}[Combinatorial p-th Calabi flow]
\begin{align}
\frac{d s_i}{dt}=(\Delta_p -K_i)(L_i-{\hat{L}}_i),~~\forall i\in V.\label{equi p-calabi_flow}
\end{align}
\end{definition}

Since all $(\Delta_p -K_i)(L_i-{\hat{L}}_i)$ are continuous function on $\mathbb{R}^{|V|}$, the flow (\ref{equi p-calabi_flow}) exists in some interval $[0,\varepsilon)$ for any initial value by Peano's existence theorem in classical ODE theory.
Next, we show that the solution of (\ref{equi p-calabi_flow}) exists for all time.

\begin{theorem}\label{Prior estimate}
For any initial value $s(0)\in \mathbb{R}^{|V|}$, the solution of combinatorial p-th Calabi flow \eqref{equi p-calabi_flow} exists for all time $t\in [0,+\infty).$
\end{theorem}
\begin{proof}
Using the inequality \eqref{bound}, then we have 
\begin{align*}
    |A_{ij}|&=\sum_{f\in F_{\{i,j\}}}-\pp{L_{i,f}}{k_j}k_j \\
    &\le\sum_{f\in F_{\{i,j\}}}\frac{2(k_f^2-1)k_j}{k_f(k_f^2-1+k_j^2)}\\
    &\le\sum_{f\in F_{\{i,j\}}} 2\frac{\frac{k_j}{\sqrt{k_f^2-1}}}{1+\frac{k_j^2}{k_f^2-1}}\\&\le|F|.
\end{align*}
Now we consider the coefficient
$$
K_i=-k_i\sum_{f\in F{\{i\}}}\frac{\partial }{\partial k_i}\mathrm{Area}(\Omega_f).
$$
Fixed a polygon $f\in F_{\{i\}}$, by Lemma \ref{gauss} and \eqref{kiestimate}, we have 
$$0<-k_i\frac{\partial\mathrm{Area}(\Omega_f)}{\partial k_i}=\left\{
\begin{array}{lc}
2k_i(k_i^2-1)^{-\frac{3}{2}}(\frac{\sqrt{k_i^2-1}}{k_f}-\arctan{\frac{\sqrt{k_i^2-1}}{k_f}})&k_i>1, \\\frac{2}{3k_f^3}&k_i=1,\\ 2k_i(1-k_i^2)^{-\frac{3}{2}}(\arctanh{\frac{\sqrt{1-k_i^2}}{k_f}}-{\frac{\sqrt{1-k_i^2}}{k_f}})&0<k_i<1.
\end{array}
\right.
$$
It is easy to know that $0<K_i\le |F|$ when $k_i=1$. Now assume that $1<k_i\le 2.$ Let $M_1=\sup_{x>0}\frac{x-\arctan{x}}{x^3}$, then we have
$$
\begin{aligned}
-k_i\frac{\partial\mathrm{Area}(\Omega_f)}{\partial k_i}&=2k_i(k_i^2-1)^{-\frac{3}{2}}(\frac{\sqrt{k_i^2-1}}{k_f}-\arctan{\frac{\sqrt{k_i^2-1}}{k_f}})\\
&\le 2k_i(\frac{k_f}{\sqrt{k_i^2-1}})^3(\frac{\sqrt{k_i^2-1}}{k_f}-\arctan{\frac{\sqrt{k_i^2-1}}{k_f}})\\
&\le 4M_1.
\end{aligned}
$$
Hence we obtain
$$
0<K_i\le 4M_1|F|.
$$
Assume that $k_i>2,$ then we have
$$
\begin{aligned}-k_i\frac{\partial\mathrm{Area}(\Omega_f)}{\partial k_i}&=2k_i(k_i^2-1)^{-\frac{3}{2}}(\frac{\sqrt{k_i^2-1}}{k_f}-\arctan{\frac{\sqrt{k_i^2-1}}{k_f}})\\
&\le\frac{2k_i}{(k_i^2-1)}\frac{k_f}{\sqrt{k_i^2-1}}(\frac{\sqrt{k_i^2-1}}{k_f}-\arctan{\frac{\sqrt{k_i^2-1}}{k_f}})\\
&\le \frac{2k_i}{(k_i^2-1)}\\
&\le 2.
\end{aligned}
$$
Hence we obtain
$$
0<K_i\le 2|F|.
$$
Assume that $0<k_i<\frac{1}{2}$. Let $M_2=\sup_{0<x<\frac{1}{2}}x\arctanh{\sqrt{1-x^2}}$, then we have
$$
\begin{aligned}
-k_i\frac{\partial\mathrm{Area}(\Omega_f)}{\partial k_i}&=2k_i(1-k_i^2)^{-\frac{3}{2}}(\arctanh{\frac{\sqrt{1-k_i^2}}{k_f}}-{\frac{\sqrt{1-k_i^2}}{k_f}})\\
&\le 2k_i(1-k_i^2)^{-\frac{3}{2}}\arctanh{\frac{\sqrt{1-k_i^2}}{k_f}}\\
&\le 3k_i\arctanh{\frac{\sqrt{1-k_i^2}}{k_f}}\\
&\le 3k_i\arctanh{\sqrt{1-k_i^2}}\\
&\le 3M_2.
\end{aligned}
$$
Hence we obtain
$$
0<K_i\le 3M_2|F|.
$$
Assume that $\frac{1}{2}\le k_i<1$. Let $M_3=\sup_{0< x\le\frac{\sqrt{3}}{2}}\frac{\arctanh x-x}{x^3},$ then we have
$$
\begin{aligned}
-k_i\frac{\partial\mathrm{Area}(\Omega_f)}{\partial k_i}&=2k_i(1-k_i^2)^{-\frac{3}{2}}(\arctanh{\frac{\sqrt{1-k_i^2}}{k_f}}-{\frac{\sqrt{1-k_i^2}}{k_f}})\\
&\le 2k_i(\frac{k_f}{\sqrt{1-k_i^2}})^3(\arctanh{\frac{\sqrt{1-k_i^2}}{k_f}}-{\frac{\sqrt{1-k_i^2}}{k_f}})\\
&\le 2M_3.
\end{aligned}
$$
Hence we obtain
$$
0<K_i\le 2M_3|F|.
$$
Now we know that there exists a constant $\mu$, depending only on the polygonal cellular decomposition of $S$, such that
$$
|A_{ij}|,|K_i|\le \mu,~~\forall i\in V, ~\forall j\sim i.
$$
Then by the definition of $\Delta_p$ and Corollary \ref{af}, we have 
$$
|(\Delta_p -K_i)(L_i-{\hat{L}}_i)|\le \mu |E|(\pi |F|+\|\hat{L}\|_{\infty})^{p-1}+\mu (\pi |F|+\|\hat{L}\|_{\infty}), ~~\forall i\in V.
$$
Hence all $|(\Delta_p -K_i)(L_i-{\hat{L}}_i)|$ are uniformly bounded by a constant. Then by the extension theorem of solutions in ODE theory, the solution exists for all time $t\ge 0$.
\end{proof}

\section{The proof of main theorem}\label{sec5}
Now we suppose that $|V|=N$ for simplicity. For a given vector $\hat{L}=({\hat{L}}_1,\cdots,{\hat{L}}_N)^T\in \mathcal{L}$, we can define a 1-form, i.e.
$$
\omega=\sum_{i=1}^{N}(L_i-{\hat{L}}_i)\mathrm{~d} s_i,
$$
it is easy to see that $\omega$ is closed. 

By Theorem \ref{main1}, there exists $\hat{s}\in \mathbb{R}^N$ which satisfies $L(\hat{s})=\hat{L}$, then we can define the potential function, i.e.
    $$
 \Theta(s)=\int_{\hat{s}}^s \omega.
$$
The function $\Theta$ is well-defined and does not depend on the particular choice of a piecewise smooth arc in $\mathbb{R}^N$ from the initial point $\hat{s}$ to $s$.

It is easy to see that $\nabla\Theta=L-\hat{L}=(L_1-{\hat{L}}_1,\cdots,L_N-{\hat{L}}_N)^T$, hence Hess $\Theta=M$. By proposition \ref{pro 3.1}, $M$ is positive definite, we know that $\Theta$ is strictly convex. 
\begin{proposition}\label{pro 3.3}
     $\hat{s}$ is the unique critical point of the potential function $\Theta$.
\end{proposition}
\begin{proof}
   By Lemma \ref{injective}, the gradient map of $\Theta$
   $$
\begin{array}{cccc}
\nabla\Theta:\mathbb{R}^N &\longrightarrow &\mathbb{R}^N\\
s&\longmapsto & L-\hat{L}\\
\end{array}
$$
is injective. 
Besides, we have $L(\hat{s})=\hat{L}$, i.e. 
$$\nabla\Theta(\hat{s})=L(\hat{s})-\hat{L}=0,$$ 
hence $\hat{s}$ is a critical point of $\Theta$. We know that $\nabla\Theta$ is injective, $\hat{s}$ is the unique critical point of $\Theta$.
\end{proof}
\begin{proposition}\label{pro 2.4}
$K=diag\{K_1,\cdots,K_N\}$ is a positive definite matrix, where $K_i (1\le i\le N)$ is defined in \ref{Ki}.
\end{proposition}

\begin{proof}
By  Lemma \ref{area_change}, it is easy to know that  $K$ is a positive definite matrix.
\end{proof}

\begin{proposition}\label{pro 4.4}
    For any $1\le i,j\le N$ and $i\ne j$, we have
    $$
    A_{ij}\ge 0,
    $$
    and $A_{ij}$ is symmetry, i.e. 
    $$A_{ij}=A_{ji}.$$
\end{proposition}

\begin{proof}
    By Lemma 2.6 in \cite{HHSZ}, we have $$\frac{\partial L_{i, f}}{\partial k_j}<0, ~~i\ne j\in V_f.$$ 
    When $j\notin V_f$, we know 
    $$
    \frac{\partial L_{i, f}}{\partial k_j}=0,~~i\in V_f.
    $$
    Hence we have  
    $$
    A_{ij}=-\frac{\partial(\sum_{f\in F_{\{i\}}}L_{i,f})}{\partial k_j}k_j=-\sum_{f\in F_{\{i\}}}\frac{\partial L_{i,f}}{\partial k_j}k_j\ge 0.$$
    Moreover,
    $$
    A_{ij}=-\frac{\partial(\sum_{f\in F_{\{i\}}}L_{i,f})}{\partial k_j}k_j
    =-\frac{\partial(\sum_{f\in F_{\{i,j\}}}L_{i,f})}{\partial k_j}k_j
    =-\sum_{f\in F_{\{i,j\}}}\frac{\partial L_{i,f}}{\partial k_j}k_j,
    $$
    We also have a similar result for $A_{ji}$, i.e.
    $$
    A_{ji}=-\sum_{f\in F_{\{i,j\}}}\frac{\partial L_{j,f}}{\partial k_i}k_i.
    $$
   In section \ref{sec2}, we know the 1-form $\omega_f=\sum_{i=1}^{|V_f|} L_{i, f} \mathrm{~d} s_i$ is closed, hence 
    $$
    \frac{\partial L_{i,f}}{\partial s_j}=\frac{\partial L_{j,f}}{\partial s_i},~~\forall i,j\in V_f,
    $$
    i.e.
    $$\frac{\partial L_{i,f}}{\partial k_j}k_j=\frac{\partial L_{j,f}}{\partial k_i}k_i,~~\forall i,j\in V_f.$$
    Hence we have $A_{ij}=A_{ji}$.
    \end{proof}
\begin{proposition}\label{pro 4.5}
    Let $g:V\rightarrow\mathbb{R}$ be a function on $V$, then we have
    $$
    \sum_{i=1}^N g_i \Delta_p g_i\le 0,
    $$
where $\Delta_p$ is defined in \eqref{lap}.
\end{proposition}

\begin{proof}
    By Proposition \ref{pro 4.4}, we know $A_{ij}=A_{ji}$.
    We exchange $i$ and $j$, it is easy to know 
    $$
    \begin{aligned}
\sum_{i=1}^N g_i \Delta_p g_i=-\frac{1}{2} \sum_{i=1}^N \sum_{j \sim i} A_{i j}\left|g_j-g_i\right|^p\le 0.
\end{aligned}
$$
\end{proof}
As a consequence, we have the following corollary, i.e.
\begin{corollary}\label{coro 2.7}
    For the total geodesic curvature $L=(L_1,\cdots,L_N)^T$ and $\hat{L}=({\hat{L}}_1,\cdots,{\hat{L}}_N)^T\in \mathcal{L}$, then we have
    $$
    (L-\hat{L})^T\Delta_p (L-\hat{L})\le 0.
    $$
\end{corollary}

\begin{proof}
    By Proposition \ref{pro 4.4} and Proposition \ref{pro 4.5}, we have 
    $$
   (L-\hat{L})^T\Delta_p (L-\hat{L})=\sum_{i=1}^N (L_i-{\hat{L}}_i) \Delta_p (L_i-{\hat{L}}_i)\le 0.
    $$
\end{proof}
We can suppose that $\{s(t)~|~t\in [0,+\infty)\}$ is a solution to the combinatorial $p$-th Calabi flow \ref{equi p-calabi_flow}. Now we define a fuction, i.e.
$$
\Phi(t)=\Theta(s(t)),
$$
then we have the following proposition, i.e.
\begin{proposition}\label{pro 4.1}
    Suppose $\{s(t)~|~t\in[0,+\infty)\}$ is a long time solution to the combinatorial
p-th Calabi flow \ref{equi p-calabi_flow}, then $\Phi(t)$ is bounded. Moreover, $\operatorname{lim}_{t\rightarrow +\infty} \Phi(t)$  exists. 
\end{proposition}

\begin{proof}
    By Proposition \ref{pro 3.3}, $\hat{s}$ is the unique critical point of the potential function $\Theta$. Moreover, $\Theta$ is strictly convex. Hence we have
    $$
    \Theta(s)\ge \Theta(\hat{s})=0, \forall s\in \mathbb{R}^N.
$$
For the function $\Phi(t)$, we have 
$$
\begin{aligned}
\Phi^{\prime}(t)&=\nabla\Theta(s(t))^T \cdot \frac{ds}{dt}=(L(s(t))-\hat{L})^T \cdot \frac{ds}{dt}\\
&=\sum_{i=1}^N (L_i(s(t))-{\hat{L}}_i)(\Delta_p-K_i(s(t))) (L_i(s(t))-{\hat{L}}_i)\\
&=\sum_{i=1}^N (L_i(s(t))-{\hat{L}}_i)\Delta_p (L_i(s(t))-{\hat{L}}_i)-\sum_{i=1}^N (L_i(s(t))-{\hat{L}}_i) K_i(s(t)) (L_i(s(t))-{\hat{L}}_i)\\
&=(L(s(t))-\hat{L})^T\Delta_p (L(s(t))-\hat{L})-(L(s(t))-\hat{L})^TK(s(t))(L(s(t))-\hat{L}).
\end{aligned}
$$
By Proposition \ref{pro 2.4} and Corollary \ref{coro 2.7}, we have 
$$
\Phi^{\prime}(t)=(L(s(t))-\hat{L})^T\Delta_p (L(s(t))-\hat{L})-(L(s(t))-\hat{L})^TK(s(t))(L(s(t))-\hat{L})\le 0.
$$
Hence $\Phi(t)$ is decreasing on [0,+$\infty$) and we know that $\Phi(t)=\Theta(s(t))\ge 0$, which implies that $\Phi(t)$ converges, i.e. $\operatorname{lim}_{t\rightarrow +\infty} \Phi(t)$ exists. Moreover, we have 
$$
0\le \Phi(t)\le \Phi(0)=\Theta(s(0)),
$$
i.e. $\Phi(t)$ is bounded.
\end{proof}

\begin{proposition}\label{pro 6.2}
Suppose $\{s(t)|t\in[0,+\infty)\}$ is a long time solution to the combinatorial
p-th Calabi flow \ref{equi p-calabi_flow} and is compactly supported in $\mathbb{R}^N$, then $\Phi(t)$ converges to 0, i.e.
$$
\operatorname{lim}_{t\rightarrow +\infty} \Phi(t)=0.
$$
\end{proposition}
\begin{proof}
Using the mean value theorem for the function $\Phi(t)$, we have 
\begin{align}
\Phi(n+1)-\Phi(n)=\Phi^{\prime}(\xi_n), ~~\forall n\in \mathbb{N},\label{mean value}
\end{align}
for some $\xi_n\in[n,n+1]$.
By Proposition \ref{pro 4.1}, $\operatorname{lim}_{t\rightarrow +\infty} \Phi(t)$  exists, hence we have 
$$
0=\operatorname{lim}_{n\rightarrow +\infty} \Phi(n+1)-\operatorname{lim}_{n\rightarrow +\infty}\Phi(n)=\operatorname{lim}_{n\rightarrow +\infty} \Phi^{\prime}(\xi_n).
$$
By the assumption, we know that $\{s(t)\}_{t\ge 0}$ is compactly supported in $\mathbb{R}^N$, hence $\{s(\xi_n)| \xi_n\in[n,n+1], n\in \mathbb{N}\}$ has a convergent subsequence, we still use $\{s(\xi_n)\}_{n\ge 0}$ to denote this convergent subsequence for simplicity. Hence there exists a point $s^*\in \mathbb{R}^N$, such that
$$
\operatorname{lim}_{n\rightarrow +\infty} s(\xi_n)=s^*.
$$
Then we have
$$
\operatorname{lim}_{n\rightarrow +\infty} L_i(s(\xi_n))=L_i(s^*)\overset{\Delta}{=}L_i^*,~~\operatorname{lim}_{n\rightarrow +\infty} A_{ij}(s(\xi_n))= A_{ij}(s^*)\overset{\Delta}{=} A_{ij}^*.
$$
Moreover,
$$
\operatorname{lim}_{n\rightarrow +\infty}K_i(s(\xi_n))=K_i(s^*)\overset{\Delta}{=} K_i^*.
$$
Then we have
$$
\begin{aligned}
\operatorname{lim}_{n\rightarrow +\infty}\Phi^{\prime}(\xi_n)&=
\operatorname{lim}_{n\rightarrow +\infty}(L(s(\xi_n))-\hat{L})^T\Delta_p (L(s(\xi_n))-\hat{L})\\
&-\operatorname{lim}_{n\rightarrow +\infty}(L(s(\xi_n))-\hat{L})^T K(s(\xi_n))(L(s(\xi_n))-\hat{L})\\
&=(L^*-\hat{L})^T\Delta_p (L^*-\hat{L})-(L^*-\hat{L})^TK^*(L^*-\hat{L})\\
&=0.
\end{aligned}
$$
By corollary \ref{coro 2.7}, we have $(L^*-\hat{L})^T\Delta_p (L^*-\hat{L})\le 0$, i.e.
$$
(L^*-\hat{L})^TK^*(L^*-\hat{L})=(L^*-\hat{L})^T\Delta_p (L^*-\hat{L})\le 0.
$$
By Proposition \ref{pro 2.4}, $K^*=diag\{K_1^*,\cdots,K_N^*\}$ is a positive definite matrix, hence we have
$$
0\le (L^*-\hat{L})^TK^*(L^*-\hat{L})\le 0,
$$
i.e. $(L^*-\hat{L})^TK^*(L^*-\hat{L})=0$, which implies that $L^*-\hat{L}=0$.

Hence we have
$$
\nabla\Theta(s^*)=L(s^*)-\hat{L}=L^*-\hat{L}=0,
$$
i.e.
$s^*$ is a critical point of $\Theta$.

By Proposition \ref{pro 3.3}, $\hat{s}$ is the unique critical point of the potential function $\Theta$, which implies that $s^*=\hat{s}$, i.e.
$$
\operatorname{lim}_{n\rightarrow +\infty}s(\xi_n)=\hat{s}.
$$
Then we have
$$
\operatorname{lim}_{n\rightarrow +\infty}\Phi(\xi_n)=\operatorname{lim}_{n\rightarrow +\infty}\Theta(s(\xi_n))=\Theta(\hat{s})=0,$$ 
which implies that
$$\operatorname{lim}_{t\rightarrow +\infty} \Phi(t)=0.$$
\end{proof}
\begin{proposition}\label{them 6.3}
    Suppose $\{s(t)|t\in[0,+\infty)\}$ is a long time solution to the combinatorial
p-th Calabi flow \ref{equi p-calabi_flow} and is compactly supported in $\mathbb{R}^N$, then $s(t)$ converges $\hat{s}$, i.e.
$$
\operatorname{lim}_{t\rightarrow +\infty} s(t)=\hat{s}.
$$
\end{proposition}
\begin{proof}
     Suppose it is not true, there exists $\delta>0$ and a sequence $\{t_n\}_{n\ge 1}$ which satisfies $\operatorname{lim}_{n\rightarrow +\infty} t_n=+\infty$, such that
     $$
    |s(t_n)-\hat{s}|>\delta,~~\forall n\ge 1,
     $$
i.e. $\{s(t_n)\}_{n\ge 1}\subset\mathbb{R}^N\setminus B(\hat{s},\delta)$. 

By the assumption, we know that $\overline{\{s(t)\}}_{t\ge 0}$ is concluded in a compact subset in $\mathbb{R}^N$, we denoted the compact subset by $D$. Hence we have 
$$
\{s(t_n)\}_{n\ge 1}\subset D\cap (\mathbb{R}^N\setminus B(\hat{s},\delta)).
$$
$\Theta$ is strictly convex and by Proposition \ref{pro 3.3}, it is easy to see that $\Theta$ has a positive lower bound on $D\cap (\mathbb{R}^N\setminus B(\hat{s},\delta))$, i.e. there exists a positive constant $\lambda$, such that 
$$
\Theta(s)\ge \lambda>0,~~\forall s\in D\cap (\mathbb{R}^N\setminus B(\hat{s},\delta)).
$$
Hence we have
$$
\operatorname{lim}_{n\rightarrow +\infty} \Phi(t_n)=\operatorname{lim}_{n\rightarrow +\infty}\Theta(s(t_n))\ge \lambda>0,
$$
which contradicts to the Proposition \ref{pro 6.2}. 
\end{proof}
Now we can proof the main Theorem \ref{flowthm}.

\begin{proof}[\textbf{Proof of Theorem \ref{flowthm}}] $2. \Rightarrow 1.$ By the assumption, there exists a generalized circle packing metric $\hat{s}$, such that $L(\hat{s})=\hat{L}$. By Theorem \ref{main1}, we know that $\hat{L}=L(\hat{s})=\nabla\mathcal{E}(\hat{s})\in\mathcal{L}$. 

$1. \Rightarrow 2.$ By Theorem \ref{Prior estimate}, we know that the solution of the combinatorial $p$-th Calabi flow (\ref{equi p-calabi_flow}) exists for all time $t\in [0,+\infty)$. Equivalently, the solution of the combinatorial $p$-th Calabi flow (\ref{p-calabi_flow}) exists for all time. We denote the solution of $p$-th Calabi flow \eqref{equi p-calabi_flow} by $s(t), t\in [0,+\infty).$ 

Since $\hat{L}\in\mathcal{L}$, by Theorem \ref{main1}, there exists a point $\hat{s}$, such that $\nabla\mathcal{E}(\hat{s})=L(\hat{s})=\hat{L}$. Hence we can construct fuctions $\Theta$ and $\Phi$, by Proposition \ref{pro 4.1}, we have 
$$
0\le \Phi(t)\le \Theta(s(0)),~~\forall t\in [0,+\infty),
$$
i.e.
$$
\{s(t)\}_{t\ge 0}\subset\Theta^{-1}[0,\Theta(s(0))].
$$
$\Theta$ is strictly convex and has a critical point $\hat{s}$, we know that $\Theta$ is a proper map. Hence $\Theta^{-1}[0,\Theta(s(0))]$ is compact, i.e.
$\{s(t)| t\in [0,+\infty)\}$ is compactly supported in $\mathbb{R}^N$. By Proposition \ref{them 6.3}, $s(t)$ converges to $\hat{s}$. Equivalently, the solution of the combinatorial $p$-th Calabi flow (\ref{p-calabi_flow}) converges to a generalized circle packing metric with the total geodesic curvature $\hat{L}$.
\end{proof}

\section{Acknowledgments}
Guangming Hu is supported by NSF of China (No. 12101275). Ziping Lei is supported by NSF of China (No. 12122119). Yi Qi is supported by NSF of China (No. 12271017). Puchun Zhou is supported by Shanghai Science and Technology Program [Project No. 22JC1400100].

\Addresses

\end{document}